\pdfoutput=1

\documentclass{amsart}

\usepackage{microtype}
\usepackage{booktabs}
\usepackage{graphicx, subfigure}
\usepackage{hyperref}

\newtheorem{thm}{Theorem}[section]

\newtheorem{lem}[thm]{Lemma}

\newtheorem*{main}{Theorem~\ref{main theorem}}

\theoremstyle{definition}
\newtheorem{defn}[thm]{Definition}

\newtheorem*{example}{Example}

\theoremstyle{remark}
\newtheorem{rmk}[thm]{Remark}

\begin{document}

\title[Unknotting and Q-theory]{The unknotting problem and\\ normal surface Q-theory}
\author{Chan-Ho Suh}
\address{
Bard High School Early College\\
30-20 Thomson Ave\\
Long Island City, NY 11101
}
\email{chanhosuh@gmail.com}
\thanks{Research was partially funded by Ryan Budney and a PIMS Postdoctoral Fellowship.}


\date{\today}

\begin{abstract}
Tollefson described a variant of normal surface theory for $3$-manifolds, called Q-theory, where only the quadrilateral coordinates are used.  Suppose $M$ is a triangulated, compact, irreducible, boundary-irreducible $3$-manifold.  In Q-theory, if $M$ contains an essential surface, then the projective solution space has an essential surface at a vertex.  One interesting situation not covered by this theorem is when $M$ is boundary reducible, e.g. $M$ is an unknot complement.  We prove that in this case $M$ has an essential disc at a vertex of the Q-projective solution space.     
\end{abstract}

\maketitle

\section{Introduction}
The problem of determining whether a given knot was unknotted has been of great interest from the earliest days of topology.  In 1954 Wolfgang Haken surprised the mathematical community by announcing at the International Congress of Mathematicians that he had an algorithm to determine if a knot was unknotted or not.  Not only was Haken an ``outsider'' of sorts, but his approach was radically different from the diagrammatic and group-theoretic techniques that preceded him\cite{epple1999}.

Recall that the unknot is the only knot that bounds a disc.  Haken's unknotting algorithm searches for such a disc in the \emph{knot complement}, the 3-manifold obtained by removing the interior of a small tube containing the knot as its core.  The knot complement is decomposed into many tetrahedra, and the decomposition gives a special system of integer linear equations.  \emph{Normal surfaces}, which intersect each tetrahedron in particularly simple pieces, are represented by some of the nonnegative integer solutions of this system.  For a system obtained from decomposing an unknot complement, Haken proved that a bounded region of the solution space contains a solution representing a disc bound by the unknot, and this leads to an unknotting algorithm \cite{haken1961}.

Haken had shown the unknotting problem was solvable in theory, but it remains a formidable challenge to make his approach viable in practice.  The chief bottleneck for Haken's algorithm is the number of variables in the integer linear programming problem.  Solving systems with only a few hundred variables can require large amounts of time and memory even on powerful computers.  Even unknot diagrams with only several dozen crossings can result in trying to solve Haken's equations with thousands of variables.

In 1995 W. Jaco and J. Tollefson took a big step toward a practical implementation of Haken's algorithm by introducing the concept of a \emph{vertex solution} and showed there existed a vertex solution representing an unknotting disc \cite{jaco-tollefson1995}.  Essentially, they considered a ``projectivized'' version of the space of solutions to Haken's equation, which is a polytope.  The vertex solutions are the vertices of this polytope.

Jaco and Tollefson's theorem is important because enumeration of the vertices of a rational polytope is a well-studied problem with simple algorithms.  The set of fundamental solutions used by Haken can be considerably larger and is not so easy to enumerate.  Furthermore, the vertex enumeration can be significantly speeded up by a filtering method by D. Letscher to eliminate vertices which cannot represent embedded surfaces, e.g. see \cite{burton2008}.

There is a version of normal surface theory called Q-theory, due to J. Tollefson \cite{tollefson1998}.  Its basis is the observation that the quadrilateral coordinates essentially determine a normal surface.  Thus we can more than halve the number of necessary normal variables by utilizing Q-theory.

Vertex enumeration, with filtering, appears to be considerably more efficient with Q-theory rather than standard normal surface theory, even when comparing similar numbers of variables \cite{burton2008}.  Intuitively, one might expect that to happen since any two quad coordinates in a tetrahedron conflict and this greatly restricts the space of normal surfaces as one enumerates vertices.  Nonetheless, a drawback of Q-theory is that with unlike Haken's theory, it has not been known if a spanning disc for the unknot is found at the vertices of the projective polytope for Q-theory.  

In this paper, we prove a normalized unknotting disc, which is \emph{minimal} in a certain technical sense, appears as a vertex in Q-theory:

\begin{main}
Let $D$ be an unknotting disc in normal form.  Suppose it minimizes $(\text{weight}(D), \sigma(D))$, where two pairs are compared lexicographically from left to right.  Then $D$ is a $Q$-vertex surface.
\end{main}

We should note that much of this paper is a reworking of the arguments of \cite{jaco-tollefson1995}.  Besides the main theorem, what is new is the generalization of Lemma 4.3  in \cite{jaco-tollefson1995} to the setting of Q-theory; in this paper, we call this the Jaco--Tollefson criterion (Theorem~\ref{jaco--tollefson criterion}).  After obtaining our version of the criterion (Theorem~\ref{q-vertex criterion}), the proof of Theorem~\ref{main theorem} then proceeds as in standard normal surface theory.  However, rather than stopping with Theorem~\ref{q-vertex criterion} and then citing the necessary results from \cite{jaco-tollefson1995} we have included a simplified and streamlined version of their arguments.  This not only makes our argument self-contained, but we think the reader will thus be prepared to tackle the more general and intricate arguments of \cite{jaco-tollefson1995}.

\section{Preliminaries}

We begin by briefly reviewing Haken's normal surface theory and then describe vertex solutions and Q-theory.

\subsection{Normal surfaces}
Let $M$ be a 3--manifold with triangulation $T$.  A properly-embedded surface $S$ is said to be \emph{normal} with respect to $T$ if $S \cap \Delta^3$ is a (possibly empty) collection of \emph{normal discs} for every tetrahedron $\Delta^3$ of $T$.  A normal disc in a tetrahedron is either a triangle, which separates one corner of the tetrahedron from the other three, or a quadrilateral (``quad''), which separates one pair of vertices from the other two.  See Figure~\ref{fig:elementary discs}.  It is clear that there are four types of triangles and three types of quads in a single tetrahedron.  Note that a normal surface may intersect a tetrahedron in more than one type of normal disc, and each type that appears may do so with multiple copies; however, a normal surface cannot intersect a tetrahedron in two different quad types, as that would cause self-intersection of the surface.

\begin{figure}
\begin{center}
\includegraphics[width=2.8in]{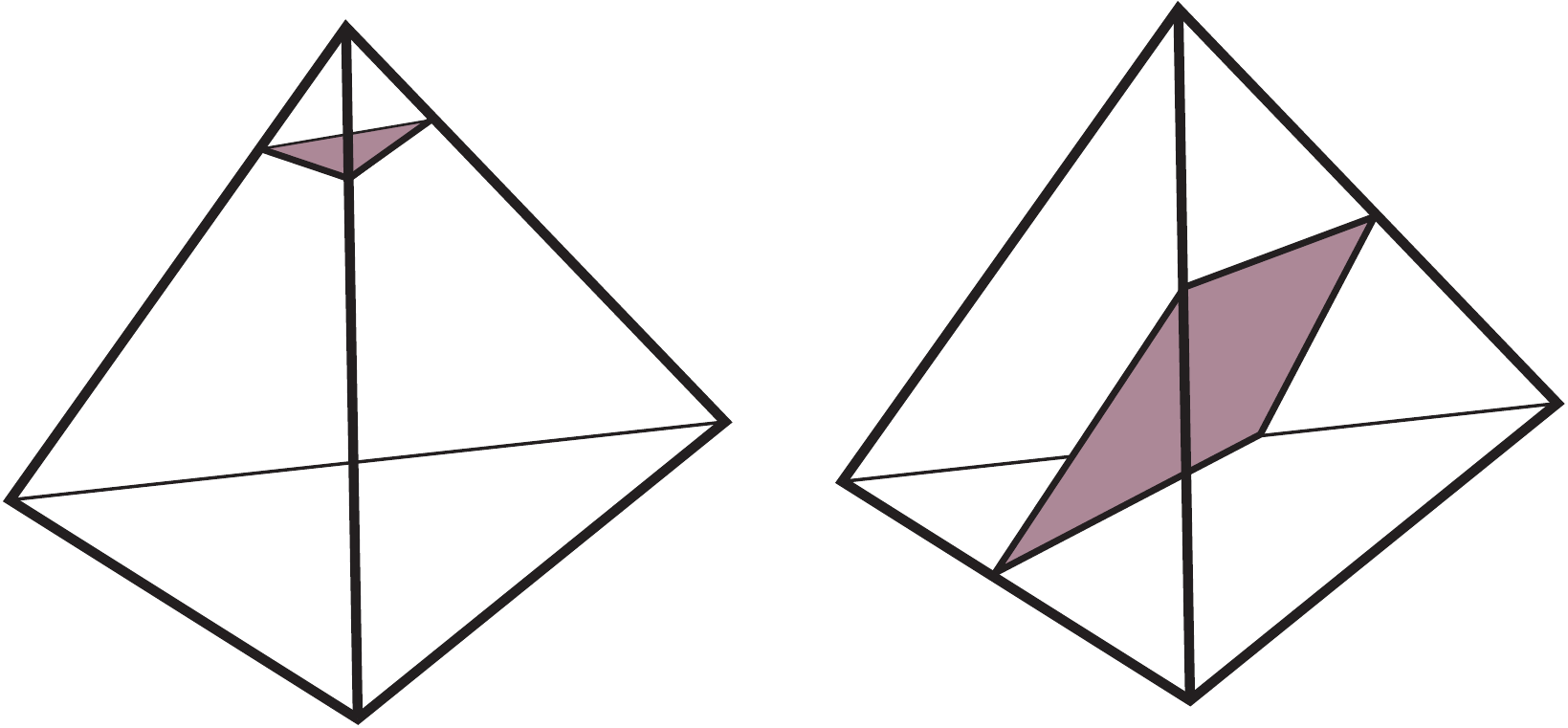}
\caption{Left - a triangle; right - a quad}
\label{fig:elementary discs}
\end{center}
\end{figure}

Haken gave a normalization procedure that inputs a compact, properly-embedded surface $S$ in a triangulated $3$-manifold $M$ and outputs a normal surface $S'$ in the same homology class as $S$ (possibly the output is the empty set).  The only properties of this procedure that we will require are: 1) when $M$ is irreducible, $S'$ is isotopic to $S$.  2) $S'$ has at most the weight of $S$. (Recall the \emph{weight} of a surface transverse to the $1$-skeleton of $M$ be the number of intersections with the $1$-skeleton.)

\subsection{Integer linear equations}
Since each normal arc type on the common face can come from two types of normal discs in each tetrahedron, a triangle and a quad  

Suppose we have numbered the normal discs for all tetrahedra in the triangulation, and $x_n$ represents the number of normal discs of the $n$-th type.  Consider a face $\Delta$ shared by two tetrahedra and suppose a normal surface passes through $\Delta$.  The normal discs abutting $\Delta$ from either either side must match.  For a normal arc type $\alpha$ on $\Delta$ and a tetrahedron with face $\Delta$, there are two normal disc types which have a normal arc of type $\alpha$ in their boundaries, a quad and a triangle.  

Thus for each interior face $\Delta$ and a normal arc type $\alpha$ on $\Delta$, we have an equation of the form $x_i + x_j = x_k + x_l$, where the $i$-th and $j$-th normal disc types belong to the same tetrahedron and have boundary containing a normal arc of type $\alpha$ and similarly for $x_k$ and $x_l$.

So normal surfaces are non-negative integral solutions to this system of integer linear equations.  The crucial observation here is that the other direction is ``almost'' true:

\begin{defn}[\bf Quad condition]
A solution vector to the matching equations satisfies the \emph{quad condition} if for every pair of coordinates corresponding to quad types in the same tetrahedron, at most one is nonzero.
\end{defn}

\begin{thm}[\bf Unique realization]
Suppose $v$ is a nonnegative, integral solution to the matching equations given by a triangulation $\mathcal T$ of $M$.  If $v$ satisfies the quad condition, then $v$ is the solution vector corresponding to a normal surface $F$, i.e. $v= v(F)$, and $F$ is unique.  
\end{thm}

\begin{rmk}
Simply put, as long as we do not have two different quad types in a tetrahedron, the obvious cause of self-intersection, we should be able to piece together the normal discs into a normal surface.  In fact, this surface is the \emph{unique} realization of the solution. 
\end{rmk}

\subsection{Haken sum}
Suppose $S$ is a normal surface.  Consider its solution vector $v(S)$ to the matching equations.  If $v(S) = v_1 + v_2$, where each $v_i$ is also a nonnegative integral solution, then $v_i = v(S_i)$ for a unique normal surface $S_i$ since each summand must satisfy the quad condition.  We say $S$ is the \emph{Haken sum} of $S_1$ and $S_2$ and write $S = S_1 + S_2$.

If $S_1$ and $S_2$ are disjoint, then $S$ is simply the disjoint union of the two surfaces.  However, it is not always the case that the $S_i$'s are disjoint.  For example, if $S$ is connected, then $S_1 \cap S_2$ cannot be empty as otherwise $S$ would be a disjoint union $S_1 \cup S_2$ and this contradicts the uniqueness of the geometric realization of $v(S)$.

Intersections occur because in order to realize non-intersecting normal discs in a tetrahedron this forces an ordering of the normal arcs of their boundaries on each face.  The orderings coming from the $S_1$ and $S_2$ in two tetrahedra adjacent to a common face may not be compatible and this is what forces the intersections.   

Consider two normal discs intersecting in a tetrahedron.  We can suppose that they intersect transversally in arcs.  An arc of intersection is \emph{regular} if there is a normal isotopy of the discs removing the arc.  In a Haken sum, the normal discs of $S_1$ and $S_2$ intersect in regular arcs; see Figure~\ref{fig:haken_sum} for an example and note that the only disc combination that will not intersect in regular arcs are two different quad types.      

An alternate view of this is to note that one can cut and paste along a regular arc to result in two normal discs of the same type(s) as before (Figure~\ref{fig:haken_sum3}).  This cut-and-paste operation is called a \emph{regular exchange}.  A \emph{regular curve} is a union of regular arcs.  Note the other way of cutting and pasting, the \emph{irregular exchange} (Figure~\ref{fig:haken_sum2}), does not give a normal surface  as a resultant disc crosses over a face more than once.

Given a Haken sum $S=S_1 + S_2$ we can, and will from now on, assume the intersection $S_1 \cap S_2$ consists of regular curves.  

Conversely, given two normal surfaces $S_1$ and $S_2$ which intersect in regular curves, $v = v(S_1) + v(S_2)$ must satisfy the quad condition.  So $v= v(S)$ for a unique $S$, and $S = S_1 + S_2$.  Geometrically, $S$ is obtained from $S_1$ and $S_2$ by doing regular exchanges along the curves of $S_1 \cap S_2$.

\begin{figure}
\begin{center}
\subfigure[A triangle and quad intersection \label{fig:haken_sum}]{
\includegraphics[width=3.6cm]{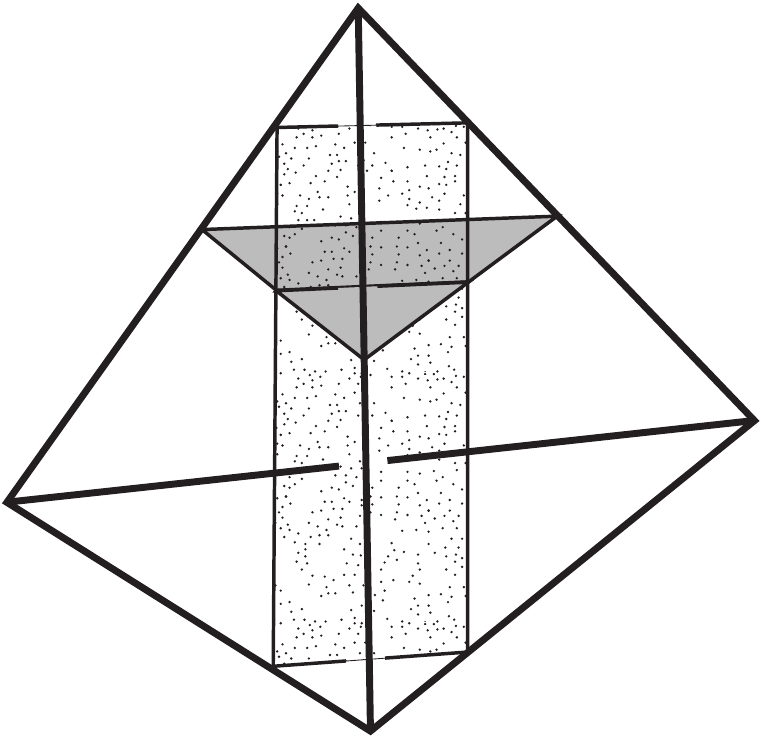}
}
\hspace{.2cm}        
\subfigure[An irregular switch \label{fig:haken_sum2}]{
\includegraphics[width=3.6cm]{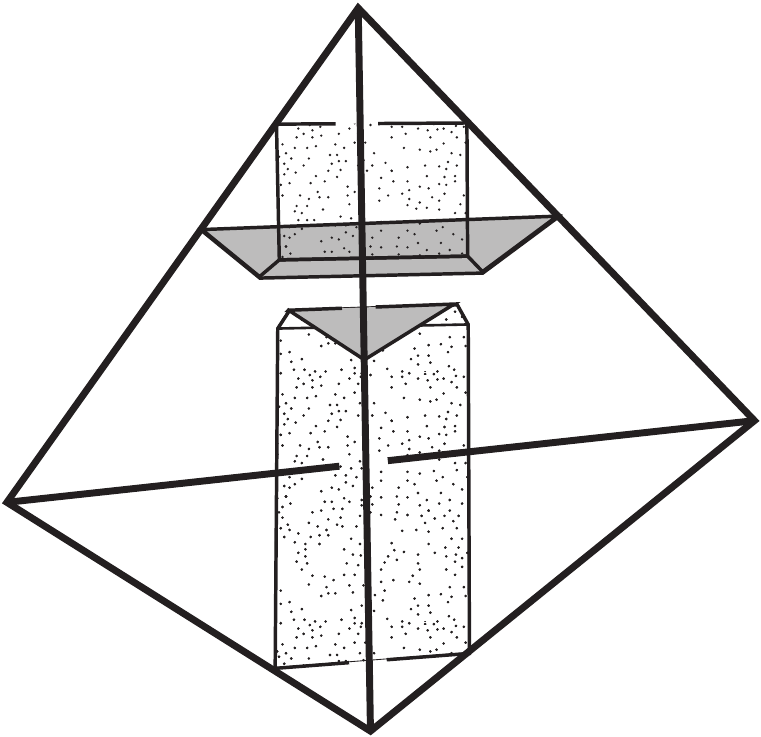}
}
\hspace{.2cm}
\subfigure[A regular switch \label{fig:haken_sum3}]{
\includegraphics[width=3.6cm]{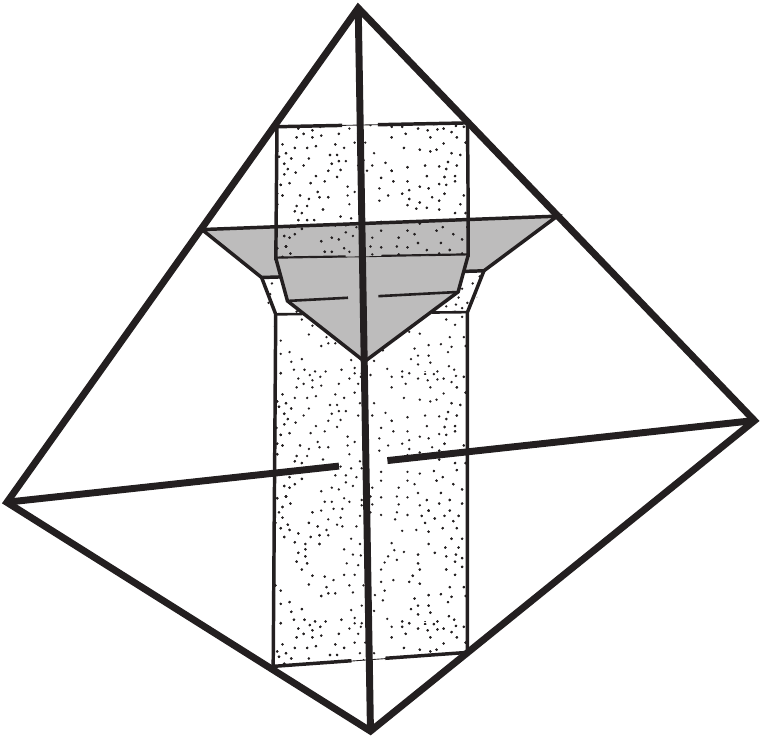}
}
\caption{Resolving intersecting normal discs by cut-and-paste}
\label{haken_sum}
\end{center}
\end{figure}

\subsection{Vertex solutions}

The space of solutions to the matching equations is a polyhedral cone in $\mathbb R^{7t}$ restricted to the positive octant ($t$ is the number of tetrahedra in the triangulation).  If we take all the rational solutions of unit Euclidean length, this is called the projective solution space, $\mathcal P$, to the matching equations.  Each projective solution can be thought of as taking a normal surface vector and normalizing it to be unit length.  

$\mathcal P$ is a convex polytope (the compact, finite intersection of closed half spaces), and its vertices are called vertex solutions.  Every vertex solution has an integral representative that is fundamental, but the converse does not hold; \cite{jaco-tollefson1995} gives examples of fundamental solutions do not project to a vertex.   

\begin{defn}[\bf Vertex solution]
An integral solution $w$ to the matching equations is a \emph{vertex solution} if it projects to a vertex $v$ of the projective solution space and $w = n\cdot v$ with $n$ the smallest positive integer for any vector satisfying these conditions.
\end{defn}

\begin{defn}[\bf Vertex surface]
Let $F$ be a normal surface.  Then $F$ is a \emph{vertex surface} if it is connected, two-sided, and $v(F)$ projects to a vertex of the projective solution space to the normal equations.
\end{defn}

The following is fairly self-explanatory:

\begin{lem}
The normal vector of a vertex surface is either a vertex solution or twice a vertex solution.  In the latter case, the vertex surface double covers a one-sided normal surface whose normal vector is a vertex solution.
\end{lem}

\begin{lem}\label{basic vertex criterion}
$F$ is a vertex surface if for every equation $nF = X + Y$ where $n$ is a positive integer $\geq2$, both $X$ and $Y$ are normally isotopic to copies of $F$ and if $F = X + Y$, then $X$ must be normally isotopic to $Y$.
\end{lem}

\subsection{Jaco and Tollefson's criterion}
For $F$ a disc, Jaco and Tollefson discovered a useful way of checking if $F$ vertex, utilizing exchange surfaces and disc patches:

\begin{defn}[\bf exchange surface]
An \emph{exchange surface} for a normal surface $F$ in $M$ is either an annulus, M\"obius band, or disc which is embedded in $M$ so that:
\begin{enumerate}
\item its interior is disjoint from $F$ and its boundary is contained in $F \cup \partial M$
\item its regular neighborhood in $M$ is orientable
\item it intersects each truncated tetrahedron $\Delta^3$ in (possibly many) 0-weight discs $E$
\item $\partial E = \alpha_1 \cup \alpha_2 \cup \gamma_1 \cup \gamma_2$, four arcs with pairwise disjoint interiors.  The $\gamma_i$'s are contained in distinct normal discs $D_1$ and $D_2$ of $F$.  The $\alpha_i$'s are contained in two different faces of $\Delta^3$.  
\item the inverse of a regular exchange exists by cutting $D_1$ and $D_2$ along $\partial E$ and regluing 
\end{enumerate}
\end{defn}

\begin{defn}[\bf exchange system]
An \emph{exchange system} for a normal surface $F$ is a collection of disjoint exchange surfaces for $F$.
\end{defn}

\begin{defn}[\bf proper exchange system]
An exchange system for $F$ is \emph{proper} if doing inverse regular exchanges across the exchange surfaces results in two families of embedded surfaces, $X$ and $Y$.  Then $F = X + Y$, where the addition is Haken sum. 
\end{defn}

\begin{rmk}
In other words, if an exchange system is not proper, it is because the result of all the inverse regular exchanges results in a self-intersecting, connected surface possibly along with other connected surfaces.
\end{rmk}

\begin{defn}[\bf disc patch]
Suppose $\mathcal A$ is an exchange system for $F$, a normal surface in $M$.  Then $D$ is a \emph{disc patch} relative to $\mathcal A$ if it is a discal component of the complement of the boundary curves of the components of $\mathcal A$ in $F$ and either: 1) its boundary, under closure, consists of one boundary component of an exchange annulus, or 2) its boundary, under closure, consists of one boundary arc from an exchange disc, and one arc in $\partial M$.
\end{defn}

\begin{rmk}
In the second case, we will abuse language by referring to the boundary arc of the exchange disc as ``bounding'' the disc patch, although technically the arc is not all of the disc patch's boundary.
\end{rmk}

\begin{defn}[\bf normally isotopic disc patches]
Two disc patches $D_1$ and $D_2$ are normally isotopic across an exchange surface if there is a normal isotopy from $D_1$ to $D_2$ keeping their boundaries in the exchange surface.  We allow the boundary $\partial D_1$ to self-intersect during this ``isotopy'' (this is only necessary for when the exchange surface is a M\"obius band).
\end{defn}

\begin{defn}[\bf adjacent disc patches]
Two disc patches are \emph{adjacent} along an exchange annulus or disc $A$ if they abut $\partial A$ on the same side (recall $A$ must be two-sided).  
\end{defn}

Now that we have defined all the terms, we can state the important criterion from \cite{jaco-tollefson1995}:

\begin{thm}[\bf Jaco--Tollefson criterion]\label{jaco--tollefson criterion}
Let $D$ be a properly-embedded normal disc.  Then $D$ is a vertex surface if and only if every exchange annulus/disc for $D$ bounds disjoint discs in $D$ which are normally isotopic.
\end{thm}

We will prove an analogue of this for Q-theory, but we will require an additional hypothesis.

\subsection{Normal surface Q-theory}
In \cite{tollefson1998}, Tollefson described a variant of normal surface theory that required only the quad coordinates.  In his Q-theory, there is one Q-matching equation for each interior edge of the triangulation of the $3$-manifold and a variable corresponding to each quad type.  

Suppose $M$ is triangulated with $t$ tetrahedra.  Then there are $3t$ quad types, which we can enumerate from 1 to $3t$.  Suppose each interior edge $e_k$ of the triangulation is oriented and we have determined a positive direction of rotation for $e_k$.  Consider a tetrahedron containing $e_k$ and a quad type $i$ that belongs to this tetrahedron.  We define $\epsilon_{k,i}$ to be 0 if this quad type avoids the edge $e_k$.  We will define it to be 1 or -1 according to the following procedure: consider the two faces $\Delta_1$ and $\Delta_2$ of the tetrahedron sharing the edge $e_k$ and suppose they are numbered according to the positive direction of rotation.  Label the tail of $e_k$ to be $a$ and the head to be $b$.  Then if we consider the normal arcs of the $i$-th quad type on these faces $\Delta_i$, either we have a normal arc which cuts off $a$ in $\Delta_1$ and then $b$ in $\Delta_2$, or we have a normal arc which cuts off $b$ in $\Delta_1$ and then $a$ in $\Delta_2$.  We define $\epsilon_{k,i} = 1$ in the former case and $\epsilon_{k,i}=-1$ in the latter case.
 
The Q-matching equations are then $\sum_{i=1}^{3t}\epsilon_{k,i}x_i = 0$ for each edge $e_k$.

\begin{defn}
A vertex-linking disc, for a boundary vertex $v$, is a connected normal surface consisting of all triangles which cut off a corner of a tetrahedron containing $v$.  Similarly, a vertex-linking sphere, for an interior vertex $v$, is a connected normal surface consisting of all triangles cutting off a corner of a tetrahedron containing $v$.
\end{defn}

\begin{thm}[\cite{tollefson1998}]
Suppose $v$ is a non-negative integral solution to the $Q$-matching equations.  Then there is a normal surface $F$ such that $v= v_Q(F)$ and $F$ has the following ``uniqueness'' property:  if there is another $F'$ such that $v_Q(F') = v$, then $F' = F + \Sigma$, where $\Sigma$ is a sum of vertex-linking discs and spheres. 
\end{thm}

As with Haken's normal surface theory, we can define the projectivized solution space of solutions to the Q-matching equations, $\mathcal P_Q$.  A connected, two-sided normal surface that is a representative of the projective class of a vertex of $\mathcal P_Q$ is called a Q-vertex surface.  

We have an analogous criterion to Lemma~\ref{basic vertex criterion} for determining if a normal surface is Q-vertex.

\begin{lem}
Suppose $F$ is a normal surface.  Then $F$ is a Q-vertex surface if and only if for any normal surfaces $X$ and $Y$ such that $n F + \Sigma = X + Y$, with $\Sigma$ a union of vertex-linking surfaces and $n$ a positive integer, then every component of $X$ and $Y$ is either normally isotopic to $F$ or a component of $\Sigma$.
\end{lem}

\section{Proof of the Main Theorem}

\begin{thm}
If a properly-embedded normal disc $D$ is Q-vertex, then for any exchange annulus or disc $A$ for $D$, $\partial A$ bounds disjoint discs which are normally isotopic across $A$.
\end{thm}

\begin{rmk}
We don't need this direction for the proof of the main theorem, but its proof is instructive.  
\end{rmk}

\begin{proof}
We prove that $D$ is not vertex if there is an exchange annulus or disc $A$ not satisfying the conditions in the theorem.

Suppose there is an exchange disc $A$ which does not satisfy the conditions.  Such a disc's boundary always bounds disjoint discs $D_1$ and $D_2$ in $D$.  So these discs are either not adjacent along $A$ or not normally isotopic across $A$.  

If they are not adjacent, then we can create two new normal surfaces, $X = A \cup (D - (D_1 \cup D_2))$ and $Y = D_1 \cup A \cup D_2$.  We can perturb $X$ and $Y$ so that they only intersect along the core curve of $A$.  (Note what we have really done is do an inverse of a regular exchange; in the future, when we do such inverse exchanges, we will take this perturbation process for granted.)  Then $D = X + Y$ and $X$ is an annulus, which cannot be normally isotopic to $D$.  So $D$ cannot be vertex. 

So suppose the discs are not normally isotopic across the exchange disc $A$.  We will use a prime $'$ to indicate an appropriate normally isotopic copy, e.g. $D_2'$ is a normally isotopic copy of $D_2$.  Then we construct normal surface $X = (D-D_1) \cup D_2'$ and $Y = (D' - D_2') \cup D_1$, where $\cup$ indicates we have glued the surfaces together appropriately using a subsurface of $A$ and perturbed them so that $X$ is transverse to $Y$.  Note we can consider these surfaces as created by swapping a disc $D_1$ for $D_2$ or vice versa.  Then $2D = X+ Y$.  If $X$ and $Y$ are normally isotopic to $D$, then $D$ contains two parallel copies of $D_1$ and two parallel copies of $D_2$.  So we can construct $X_1$ and $Y_1$ from $D$ by swapping the parallel copies of $D_1$ for $D_2$ and vice-versa.  If yet again $X_1$ and $Y_1$ are normally isotopic to $D$, then there are four parallel copies of $D_1$ and $D_2$ in $D$, and we can form $X_2$ (resp. $Y_2$) by swapping the four parallel copies of $D_1$ for four copies of $D_2$.  Eventually this must terminate, so eventually we end up with a sum $2D = X_n + Y_n$ where each summand is not normally isotopic to $D$.      

Now that we've disposed of the case when $A$ is an exchange disc consider when $A$ is an exchange annulus.  The boundary of $A$ bounds two discs $D_1$ and $D_2$ in $D$.  Suppose they are not adjacent along $A$.  Then we have two cases.

\textbf{Case 1:}  \underline{$D_1$ and $D_2$ are disjoint}.  We can create two new normal surfaces by an inverse regular exchange along $A$, $X = A \cup (D - (D_1 \cup D_2))$ and $Y = D_1 \cup A \cup D_2$.  Then $D = X + Y$.  $Y$ is the union of two discs and is topologically a sphere, so cannot be normally isotopic to $D$.  Thus $D$ is not vertex.  

\textbf{Case 2:}  \underline{$D_1$ and $D_2$ are not disjoint}.  Suppose $D_1 \subset D_2$.  Let $D'$ be a parallel copy of $D$ and $D'_i$ a parallel copy of $D_i$ in $D'$. Define the normal surface $X = (D - D_1) \cup A' \cup (D'- D_2')$, where $A'$ is an annulus in $A$ that connects the two surfaces' boundaries.  Similarly define $Y = D_1 \cup A''\cup (D_2' - D_1') \cup A'''\cup D_1'$ where $A''$ and $A'''$ are appropriate annuli in $A$ connecting the surfaces.  Then $2D = D + D'= X + Y$.  Again we arrive at the conclusion that $D$ cannot be vertex by seeing that $Y$ is topologically a sphere.

So we can suppose $D_1$ and $D_2$ are discs adjacent along $A$.  If one is contained in the other, say $D_1 \subset D_2$, then as before, an inverse regular exchange will create $D$ as a Haken sum of surfaces that are not normally isotopic to $D$.  So suppose they are disjoint.  Then we can construct normal surfaces $X$ and $Y$ by swapping the discs $D_i$.  Define $X = (D-D_1) \cup D_2'$ s.t. $X\cap D = \partial D_2$ and $D_2'$ is a parallel copy of $D_2$, and $Y = (D' - D_2') \cup D_1$; we can suppose $X$ and $Y$ intersect only along the core curve of $A$.  Then $2D = X+ Y$.  If $D_1$ is not normally isotopic to $D_2$, we can use the maximal disc swapping argument as in the exchange disc case to conclude that neither $X$ nor $Y$ are isotopic to $D$.
\end{proof}

\begin{thm}\label{q-vertex criterion}
Suppose $D$ is a properly-embedded normal disc in an irreducible, triangulated $3$-manifold $M$ and for any exchange annulus or disc $A$ for $D$, $\partial A$ bounds disjoint discs which are normally isotopic across $A$.  Then $D$ is a Q-vertex surface.
\end{thm}

\begin{rmk}
The irreducibility assumption is helpful in extending the Jaco--Tollefson criterion to the setting of Q-theory.  We don't know if it is necessary.  Jaco and Tollefson did not require irreducibility for their criterion, but they did require it for their version of our main theorem.
\end{rmk}

\begin{proof}
We prove the contrapositive: if $D$ is not a vertex surface, then there exists an exchange annulus or disc for $D$ such that its boundary does not bound disjoint disks which are normally isotopic along the exchange surface.  

Since $D$ is not vertex, $nD + \Sigma = X + Y$, where neither $X$ nor $Y$ have components normally isotopic to multiples of $D$, $n$ is a positive integer, and $\Sigma$ is a union of vertex-linking surfaces.  We can pick $X$ and $Y$ to have minimal intersection with respect to this property.  Note this implies each $X$ and $Y$ is connected.  For, if $X = X_1 \cup X_2$, a disjoint union, then $Y$ must intersect both $X_1$ and $X_2$.  Otherwise, an $X_i$ would be a union of copies of $D$ and/or vertex-linking surfaces and we could subtract $X_i$ from each side of the original equation.  Since $Y$ intersects both $X_i$'s, we can write $nD+ \Sigma = X_1 + Y'$, where $Y' = X_2 + Y$ and $X_1 \cap Y'$ has fewer intersections than $X \cap Y$.   

Consider a proper exchange system for the Haken sum $X+Y$.  Note that we can suppose no component of this system is a M\"obius band.  For otherwise, we can create an exchange annulus which is locally on both sides of the M\"obius band, and this annulus is the annulus we seek: it cannot bound disjoint discs at all since its two boundary components cobounds an annulus in a single copy of $D$.

Let $A$ be an exchange surface in the exchange system which has one boundary curve $\alpha_1$ bounding a disc patch $E_1$, which we consider to lie in $D$.  The other boundary curve of $A$, $\alpha_2$, must bound a disc $E_2$ in a copy of $D$ (which we call $D'$) or a component of $\Sigma$ (which we call $F$).    

We first dispose of the latter situation: $E_2 \subset F$.  If $A$ is a disc (and so necessarily is $F$), then pick $E_2$ to be adjacent.  If the $E_i$'s are normally isotopic, then since $E_1$ is a disc \emph{patch} and components of $\Sigma$ do not have exchange surfaces between them, $E_2$ is a disc patch also.  But then a disc swap between $X$ and $Y$ results in surfaces which still sum to $nD + \Sigma$, but with one less exchange surface, which contradicts minimality of $X \cap Y$.   

So now we can suppose the $E_i$'s are not normally isotopic.  Suppose they are, however, adjacent.  First note that $E_2$ must also be a disc patch.  This follows from irreducibility of $M$; since $E_1 \cup A \cup E_2$ bounds a $3$-ball, it cannot contain any part of $D$, so the only exchange surfaces inside this $3$-ball are ones connecting a disc in the interior of $E_2$ to a component of $\Sigma$.  But there are no exchange surfaces connecting components of $\Sigma$.  So $E_2$ is a disc patch. 

Now by swapping the $E_i$'s in $X$ and $Y$ to obtain $\widehat{X}$ and $\widehat{Y}$ resp., we now have $nD + \Sigma = \widehat{X} + \widehat{Y}$.  In order not to contradict minimality of the intersection $X\cap Y$, $\widehat{X}$ must be normally isotopic to a multiple of $D$, which means that $D$ must contain a copy of $E_2$, $E_2'$.  But the exchange surface $A$ spans $E_1$ and the vertex-linking surface $F$,  and since $D$ must be disjoint from $F$ (and all of $\Sigma$), this implies $E_1$ is in the parallel region between $E_2$ and $E_2'$.  This is a contradiction, since this would imply $E_1$ and $E_2$ are normally isotopic.  

It remains to consider when $E_i$'s are not adjacent.  If $F$ were a vertex-linking sphere, we could pick $E_2$ to be adjacent to $E_1$ along $A$ and use the previous argument.  So suppose $F$ is a vertex-linking disc.  Now by doing an inverse regular exchange along $A$, we obtain $D$ as a Haken sum of a sphere, $E_1 \cup A\cup E_2$, and an annulus $Z$, which intersects along the core curve of $A$.  But the sphere must bound a $3$-ball on one side; however it's clear the $3$-manifold on either side of it has a boundary component of $M$ ($Z$ has a boundary component on each).  

Thus from now we suppose $E_2 \subset D'$, with possibly $D'=D$.  If $E_1$ and $E_2$ are not adjacent, then: if 1) $E_2$ is in $D$: we are done; 2) $E_2$ is in a copy of $D$, $D'$: then extend $A$ past $D'$ to $D$ so that $A$ now bounds discs in $D$ which are not adjacent.  Then we are done.

So we can suppose $E_i$'s are adjacent along $A$.  If they are not disjoint, then we are also done.  So suppose also that they are disjoint.

If $E_1$ is not normally isotopic to $E_2$, then we are done if $D=D'$.  Otherwise, if $E_2\subset D'$, then we can extend $A$ as before to finish.

Thus we now suppose that every disc patch $E$ in a copy of $D$ given by our exchange system, has the property that $E'$, the disc given by following the exchange surface bounding $E$, is an adjacent disc (but not necessarily a patch), which is normally isotopic to $E$.  Furthermore by a prior argument, we can suppose that $E'$ is in a copy of $D$, i.e. not in a vertex-linking surface.  We will now derive a contradiction from these assumptions to complete the proof.

Pick $E_1$ to be a disc patch of least weight amongst all disc patches in copies of $D$.  Let $E_1'$ be the corresponding disc, as above.  If $E_1'$ is not a disc patch also, then look inside it for a disc patch $E_2$.  Corresponding to $E_2$ along an exchange annulus or disc $A_2$ is $E_2'$.  If this is not a disc patch also, then look inside $E_2'$ for a disc patch $E_3$ which is connected along $A_3$ to $E_3'$.  Since there are finitely many exchange surfaces this must eventually terminate.  Termination either means we have found a disc patch $E_i$ such that $E_i'$ is also a disc patch, or the process has cycled.  

If cycling occurs, we can suppose (after renumbering), that there are $n$ discs $E_i$ and that $E_{n+1}' = E_1$.  Recall that $E_1$ was least weight amongst all disc patches in copies of $D$.  Since $E_i'$ is normally isotopic to $E_i$, all together this implies the weight of all the $E_i$'s and $E_i'$'s are the same.  Note that all the $A_i$'s are either discs or annuli.

By gluing together all the disc or annuli of the form $E_i'-E_{i+1}$ we obtain a zero weight annulus, M\"obius band, torus, or Klein bottle $G$.  Consider a $2$-simplex $\Delta$ which contains a component $C$ of the intersection of $G$ with the $2$-skeleton.  Pick an orientation for $C$.  Since $C$ is a simple closed curve made of arcs $\alpha_j$ coming from $A_i \cap \Delta$ and alternating with subarcs $d_j$ of $D \cap \Delta$, the orientation of $C$ gives one for the $\alpha_j$'s and $d_j$'s.   

For each arc $d_j$ consider the side of $\Delta$ that it is pointing to.  As $j$ increases and we go around $C$, $d_j$ and $d_{j+1}$ can point at different sides.  To make the discussion simpler, we suppose we round the corners of $C$ slightly to obtain a smooth curve.  Then the tangent vector along $d_j$, but away from the ends, points in the same manner as $d_j$.  As we go around $C$, the tangent vector must make one complete rotation, possibly back-tracking at times, either clockwise or counter-clockwise.  
 
Since $E_i$ and $E_i$' are adjacent along $A_i$, the discs/annuli we glued together along the $A_i$ to obtain $G$ are not adjacent along $A_i$.  In particular, $d_j$ and $d_{j+1}$ are locally on opposite sides of $\alpha_j$.  Now we consider how the tangent vector changes as it moves from $d_j$ to $d_{j+1}$.  Either the tangent vector along $d_j$ points in the same direction as that along $d_{j+1}$ or differs by 1/3 of a complete revolution.  The tangent vector along $C$ must make consecutive 1/3 rotations in the same direction as otherwise it will never make a full rotation along $C$.  But it is easy to see that after one such 1/3 rotation, there is no room to make another 1/3 rotation in the same direction.  So this is a contradiction.

Therefore eventually we must find a disc patch $E_i$ such that $E_i'$ is also a disc patch.  Since they are normally isotopic, we can swap these discs between $X$ and $Y$ to obtain $\widehat{X}$ and $\widehat{Y}$ such that $|\widehat{X} \cap \widehat{Y}| < |X\cap Y|$  while still having $\widehat{X} + \widehat{Y} = nD + \Sigma$.  This contradicts minimality of the intersection.
\end{proof}

\begin{defn}[\bf Bad disc]
Let $\Sigma$ be a normal surface and $A$ an exchange annulus or disc for $\Sigma$.  Suppose $A$ bounds disjoint adjacent disc patches, $D_1$ and $D_2$.  A \emph{bad disc} in a 2-simplex $\Delta$ relative to $D_1$ is a component of $\Delta - (D_1 \cup A)$ which has 4 sides alternating between $A$ and $D_i$'s.
\end{defn}

\begin{defn}[\bf weight]
The \emph{weight} of a surface transverse to the $1$-skeleton of a triangulated $3$-manifold is the number of points of intersection with the 1-skeleton.
\end{defn}

\begin{lem}\label{no bad disc}
Suppose $A$ is an exchange annulus/disc for a disc $D$ with $\partial A$ bounding disjoint adjacent discs, $D_1$ and $D_2$.  If $D_1$ and $D_2$ are not normally isotopic along $A$ and weight of $D_1$ is minimal with respect to all such choices for $A$ and $D_i$'s, then there are no bad discs relative to $D_1$.  
\end{lem}

\begin{proof}
Let $C$ be a bad disc for $D_1$.  Compress $A$ along $C$ into $A_1$ and $A_2$, which are either both annuli or one annulus and one disc.  There may be pieces of an $A_i$ in a tetrahedron that are isotopic, keeping boundary in $D$, into a bad disc in a face of the tetrahedron.  We assume that the $A_i$ are isotoped through any such bad discs and repeat this as necessary.  It is not difficult to see the result is that the $A_i$'s are exchange surfaces.  The only thing to check is that each component of intersection of $A_i$ with a tetrahedron spans distinct faces.  

Consider the discs $D_1'$ and $D_2'$ bound by $\partial A_1$ in $D$, and the discs $D_1''$ and $D_2''$ bound by $\partial A_2$ in $D$.  $D_1'$ and $D_1''$ have less weight than $D_1$ since their union is $D_1$.   Similarly, $D_2'$ and $D_2''$ have less weight than $D_2$.  Suppose $D_1'$ and $D_1''$ are normally isotopic along $A_1$.  Then $D_2'$ and $D_2''$ cannot be normally isotopic along $A_2$.  Otherwise, we can extend the normal isotopy to one of $D_1$ to $D_2$ along $A$.  But this contradicts that we chose $D_1$ to be the least weight disc patch not normally isotopic along an exchange annulus or disc to an adjacent disc patch.
\end{proof}

\begin{defn}[\bf Face fold]
Let $F$ be a normal surface and $A$ an exchange surface for $F$.  Then a \emph{face fold} with respect to $F$ and $A$ is a component of $\Delta - (F \cup A)$ such that its boundary (in $\Delta$) is composed of $a \cup \gamma_1 \cup \gamma_2 \cup e$, where $a$ is a component of $A \cap \Delta$, $\gamma_i$ is a subarc of a component of $F\cap \Delta$, and $e$ is a subarc of an edge of $\Delta$.  An \emph{innermost} face fold is a face fold such that neither $\gamma_i$ has interior intersecting a component of $A \cap \Delta$.  If a disc patch $E$ intersects a face $\Delta$ so that the intersection contains part of the boundary for an innermost face fold, then we say $E$ \emph{lies on a face fold}.
\end{defn}

\begin{lem}\label{face fold}\label{no same arc}
Suppose $A$ is an exchange annulus/disc for a disc $D$.  Suppose $\partial A$ bounds adjacent disc patches $D_1$ and $D_2$ which are not normally isotopic along $A$.  Suppose $D_1$ has the least weight for all choices of exchange annuli or discs $A$ for $D$ that have this property.  Then:

\begin{enumerate}
\item $D_1$ does not lie on a face fold.
\item For each $2$-simplex $\Delta$, no component of $A\cap \Delta$ has endpoints in two normal arcs of $D\cap \Delta$ of the same type.
\end{enumerate}
\end{lem}

\begin{proof}
If $D_1$ lies on a face fold, then we can isotope $A$ by pushing it through part of the edge in the boundary of this face fold.  This reduces the weight of $D_1$ while preserving its other properties, which contradicts we picked $D_1$ to be least weight.

Let $\Delta$ be a face. Suppose $\lambda_1$ and $\lambda_2$ are two normal arcs of the same type in $D\cap \Delta$, and $a$ is an arc of $A\cap \Delta$ spanning the $\lambda_i$'s.  If we consider a small neighborhood of the component of $\partial A$ which bounds $D_1$, we see that it must stay locally on the same side of $D$.  Similarly for $D_2$.  Thus any arc of $A\cap \Delta$ which touches either $\lambda_i$ must be between the $\lambda_i$'s.  If $a$ is the only such arc, then this contradicts the earlier fact that $D_1$ does not lie on a face fold.  If there is another such arc, then this creates a bad disc.  But this contradicts Lemma~\ref{no bad disc}.  
\end{proof}

\begin{defn}[\bf size]
Define the \emph{size} of a normal surface, $\sigma(F)$, to be the number of nonzero normal coordinates.
\end{defn}

\begin{lem}\label{swap lemma}
Suppose $A$ is an exchange annulus (resp. exchange disc) for a disc $D$.  Suppose $\partial A$ bounds adjacent disc patches $D_1$ and $D_2$ which are not normally isotopic along $A$.  Suppose $D_1$ has the least weight for all choices of exchange annuli or discs $A$ for $D$ that have this property.  Let $D'$ be defined as the normal disc obtained from $D$ by replacing $D_2$ and all subdiscs of $D$ which are parallel to $D_2$ by copies of $D_1$.  

If $D_1$ and $D_2$ have the same weight and $D$ and $D'$ have the same size, then there is an extension of $A$ to a zero weight annulus (resp. disc) $A'$ so that $\partial A'$ bounds $D_1$ and $D_3$, where $D_3$ is a disc in $D'$ not normally isotopic to $D_1$ along $A'$.  This extension is a union of exchange annuli (resp. discs) connecting copies of $D_1$ and in particular, a copy of $D_1$ with $D_3$.  
\end{lem}

\begin{proof}
By Lemma~\ref{no same arc} arcs of $A\cap \Delta$ span different normal arcs types.  In addition, there are no bad discs relative to $D_1$ or $D_2$ (by Lemma~\ref{no bad disc}) and $D_1$ doesn't lie on a face-fold (by Lemma~\ref{face fold}).  These last two conditions imply that there can only be at most one arc of $A\cap \Delta$ spanning a given pair of normal arc types.    

Suppose we have such an arc $a$ which spans $\lambda_1$ and $\lambda_2$ and let $\lambda_i$ correspond to $D_i$.  Since $D_1$ does not lie on a face fold, after the switch of $D_2$ to a copy of $D_1$, $\lambda_2$ becomes of the same type as $\lambda_1$.  Indeed, each $\lambda_i$ was part of a normal disc $E_i$ with the $E_i$'s of distinct types, and after the switch, $E_2$ becomes of the same type as $E_1$.  By hypothesis, the size of $D$ does not change because of the switch of $D_2$ for a copy of $D_1$.  Thus there must be another normal disc of the same type as $E_2$ before the switch.

We can extend the component of $A$ containing $a$ in the tetrahedron to span to this other normal disc, and evidently we can do this for every component of $A$ in every tetrahedron.  Therefore, starting with $D$, we can switch each copy of $D_2$ to a copy of $D_1$, one at a time, and perform this extension each time, to obtain an annulus $A'$.

$A'$, the extension of $A$, still has one boundary bounding $D_1$, but now the other bounds a disc $D_3$ which cannot be normally isotopic to $D_2$.  If $D_3$ and $D_1$ are not adjacent along $A'$, then certainly they cannot be normally isotopic along $A'$.  If they are adjacent, then since $D_2$ is between $D_1$ and $D_3$, $D_3$ cannot be normally isotopic to $D_1$ either.  
\end{proof}

\begin{thm}\label{main theorem}
Suppose $D$ is a properly-embedded normal disc for an irreducible, triangulated $3$-manifold $M$ such that $\partial D$ is not nullhomotopic in $\partial M$.  If $D$ minimizes the pair $(\operatorname{weight}(D), \operatorname{size}(D))$ over all such discs, then $D$ is a vertex surface.
\end{thm}

\begin{proof}
Suppose such a $D$ is not a vertex surface.  Then by Lemma~\ref{q-vertex criterion}, $D$ has an exchange annulus or disc $A$ such that $A$ either doesn't bound disjoint adjacent discs, or the disjoint adjacent discs bound by $\partial A$ are not normally isotopic along $A$.

We want to apply Lemma~\ref{swap lemma}.  So we need to show that $\partial A$ bounds adjacent discs.  Suppose $A$ is an exchange disc.  By definition of an exchange \emph{disc}, $\partial A$ meets $D$ in two properly embedded arcs which split off disjoint discs $D_1$ and $D_2$.  If $D_1$ and $D_2$ are not adjacent, an inverse regular exchange along $A$ results in a disc $\widehat{D}$ and an annulus $F$.  If $\partial \widehat{D}$ is nullhomotopic in $\partial M$, then since $M$ is irreducible, then $\widehat{D} \cup E$, where $E$ is the disc in $\partial M$ bound by $\partial \widehat{D}$, bounds a $3$-ball.  But we can use $F$ to find a closed curve puncturing $\widehat{D} \cup E$ only once, which means the sphere does not separate, a contradiction.  Thus $\widehat{D}$ must be essential, but it has lesser weight than $D$, which contradicts that we picked $D$ to be the least weight essential disc.    

So when $A$ is a disc, it bounds adjacent discs.  Suppose $A$ is an exchange annulus.  Then $\partial A$ bounds discs $D_1$ and $D_2$ on $D$.  Suppose they are not disjoint, say $D_1 \subset D_2$.  Then we can remove $D_2$ from $D$ and join $D_1$ to the remaining part of $D$ using $A$.  If this new disc is not normal, then we can use the normalization procedure to result in a normal disc.  In any case, the new normal disc has less weight than $D$, since $D_1$ has less weight than $D_2$.

Thus it must be that $D_1 \cap D_2 = \emptyset$.  If they are not adjacent along $A$, then $D_1 \cup_A D_2$ is a normal sphere.  It cannot separate since $D - (D_1 \cup D_2)$ is on either side of the sphere.  So $D_1$ and $D_2$ are adjacent along $A$.  

According to our starting assumption, $D_1$ and $D_2$ are not normally isotopic along $A$.  Now these two discs have the same weight.  Otherwise we could swap the one of greater weight for a copy of the other, to create a disc of less weight than $D$.  In the case of $A$ an exchange disc, we have to take care that the swap keeps the boundary essential in $\partial M$.  But this is easy to establish; consider $D_1 \cup A \cup D_2$.  If its boundary is not nullhomotopic in $\partial M$, then after normalization we have a normal essential disc of lesser weight.  So the boundary of $D_1\cup A\cup D_2$ must bound a disc in $\partial M$ and clearly the swap of $D_2$ for $D_1$ will result in an essential disc.

So if we switch all parallel copies of $D_2$ for copies of $D_1$ to obtain a disc $D'$, we preserve weight.  In addition, the size of $D'$ must be the same as $D$, since we minimized size in the weight class of $D$.  By picking $D_1$ to be least weight over all discs bound by an exchange annulus that are not normally isotopic, we can apply Lemma~\ref{swap lemma}, to conclude that $A$ can be extended to an annulus $A'$ for $D'$ so that $\partial A'$ bounds $D_1$ and $D_3$, where $D_3$ is a disc not normally isotopic to $D_1$.  There is an exchange annulus $A''$ contained in $A'$ such that $\partial A''$ bounds $D_3$ and a parallel copy of $D_1$.  We can repeat the above argument to deduce that $D_3$ is adjacent along $A''$ to this copy of $D_1$, switch all parallel copies of $D_3$ for $D_1$ to obtain $D''$, and then apply Lemma~\ref{swap lemma}.  Our hypotheses ensure we can continue extending the exchange annulus and finding discs that are not normally isotopic to $D_1$.  However, this process needs to stop, since we cannot keep extending the annulus without eventually running out of normal discs.  This contradiction completes the proof.
\end{proof}

\end{document}